\documentclass[12pt]{amsart}
\usepackage{amsmath,amsthm,amsfonts,amssymb}
\usepackage[all]{xy}
\usepackage{comment}

\DeclareMathOperator{\Rep}{Rep}
\DeclareMathOperator{\End}{End}

\DeclareMathOperator{\Hom}{Hom}
\DeclareMathOperator{\im}{i}

\renewcommand{\dim}{\operatorname{dim}}

\newcommand{\uHom}{\underline{\Hom}}

\newcommand{\ve}{\varepsilon}
\newcommand{\FPdim}{{\rm FPdim}}

\newcommand{\C}{\mathbb C}
\newcommand{\Z}{\mathbb Z}
\newcommand{\Q}{\mathbb Q}
\newcommand{\F}{\mathbf{F}}
\newcommand{\CC}{\mathcal{C}}
\newcommand{\E}{\mathcal{E}}
\newcommand{\D}{\mathcal{D}}
\newcommand{\TY}{\mathcal{TY}}
\newcommand{\N}{\mathbb N}

\newcommand{\ot}{\otimes}

\newcommand{\B}{\mathcal{B}}

\newcommand{\op}{\oplus}
\newcommand{\om}{\omega}
\newcommand{\lan}{\langle}
\newcommand{\ra}{\rangle}
\newcommand{\la}{{\lambda}}

\newcommand{\al}{\alpha}
\newcommand{\one}{\mathbf{1}}

\newcommand{\GT}{\mathcal{GT}}
\newcommand{\ts}{\tilde{s}}

\newcommand{\zero}{{\mathbf 0}}
\newcommand{\ga}{\gamma}

\newcommand{\g}{\mathfrak{g}}
\newcommand{\so}{\mathfrak{so}}
\newcommand{\ssl}{\mathfrak{sl}}
\newtheorem{thm}{Theorem}[section]

\newtheorem*{conjno}{Conjecture}
\newtheorem{lem}[thm]{Lemma}
\newtheorem{prop}[thm]{Proposition}
\newtheorem{cor}[thm]{Corollary}
\theoremstyle{definition}
\newtheorem{defn}[thm]{Definition}
\newtheorem{example}[thm]{Example}
\newtheorem{conj}[thm]{Conjecture}
\newtheorem{question}[thm]{Question}
\newtheorem{remark}[thm]{Remark}

\begin{document}

\title[A finiteness property]
{A finiteness property for braided fusion categories}
\author{Deepak Naidu}
\address{Department of Mathematics,
    Texas A\&M University, College Station, TX 77843, USA}
\email{dnaidu@math.tamu.edu}
\author{Eric C. Rowell}
\address{Department of Mathematics,
    Texas A\&M University, College Station, TX 77843, USA}
   \email{rowell@math.tamu.edu}
\thanks{The second author is partially supported by NSA grant H98230-08-1-0020.
 We thank Dmitri Nikshych, Victor Ostrik and Michael Larsen for useful discussions.}
\begin{abstract}
We introduce a finiteness property for braided fusion categories, describe a conjecture that would characterize categories possessing this, and verify the conjecture in a number of important cases.  In particular we say a category has \emph{property} \textbf{F} if the associated braid group representations factor over a finite group, and suggest\index{\footnote{}} that categories of integral Frobenius-Perron dimension are precisely those with property \textbf{F}.
\end{abstract}

\maketitle

\section{Introduction}
Given an object $X$ in a braided fusion category $\CC$ one may construct a family of braid group representations via the homomorphism $\C\B_n\rightarrow \End(X^{\ot n})$ defined on the braid group generators $\sigma_i$ by
$$\sigma_i\rightarrow Id_X^{\ot i-1}\ot c_{X,X}\ot Id_X^{\ot n-i-1}$$
where $c_{X,X}$ is the braiding on $X\ot X$.  In this paper we consider the problem of determining when the images of these representations are finite groups.
We will say a category $\CC$ has \textit{property $\mathbf{F}$} if all such braid representations factor over finite groups.  Various cases related to quantum groups at roots of unity, Hecke and BMW algebras, and finite group doubles have been studied in the literature, see \cite{ERW,FRW,FLW,jones86,jonescmp,LR,LRW}. The evidence found in these papers partially motivates (see also \cite[Section 6]{RSW}):
\begin{conjno}
 A braided fusion category $\CC$ has property $\mathbf{F}$ if, and only if, the Frobenius-Perron dimension $\FPdim(\CC)$ of $\CC$ is an integer, (\textit{i.e.} $\CC$ is \emph{weakly integral}).
\end{conjno}
In Section \ref{1:propf} we provide further details and some preliminary evidence supporting the conjecture.  For the moment we state an example \cite{FLW, jones86} (associated with quantum groups of type $A$) which supports the conjecture.
The braided fusion category $\CC(\mathfrak{sl}_2, q, \ell)$ (see Section~\ref{2:qgcats} for notation)
has property $\mathbf{F}$ if, and only if, $\ell \in \{2,3,4,6\}$. 
On the other hand,  $\CC(\mathfrak{sl}_2, q, \ell)$ is weakly integral
if, and only if, $\ell \in \{2,3,4,6\}$. For $\ell=4,6$ these categories are non-integral,
possessing simple objects of dimension $\sqrt{2}$ and $\sqrt{3}$ respectively.

Without a fairly explicit description of the algebras $\End(X^{\ot n})$ and the action of $\B_n$, verifying that a given braided fusion category $\CC$ has property $\F$ is generally not feasible.  Even if such a description is available, determining the size of the image can be difficult task.  On the other hand, showing that $\CC$ fails to have property $\F$ can sometimes be done with less effort, as one need only show that the image of $\B_3$ is infinite.  Assuming that $X^{\ot 3}$ has at most 5 simple subobjects, knowledge of the eigenvalues of $\sigma_1$ is essentially all one needs to determine if the image of $\B_3$ is infinite: criteria are found in \cite{RT}.  This is particularly effective for ribbon categories associated with quantum groups, see \cite{jones86,FLW,LRW}.

Verifying property $\F$ becomes more manageable under the stronger hypothesis that $\FPdim(X)\in\N$ for each $X$, \emph{i.e.} for \emph{integral} braided fusion categories $\CC$.  By \cite[Theorem 8.33]{ENO} any integral fusion category is $\Rep(H)$ for a finite dimensional semisimple quasi-Hopf algebra $H$.  In this paper we focus on verifying property $\F$ under this additional hypothesis, making use of \cite[Corollary 4.4]{ERW}: \textit{braided group-theoretical fusion categories have property $\F$.} We do not consider the ``only if'' direction of the conjecture here.

There are two main sources of weakly integral braided fusion categories in the literature: Drinfeld centers of Tambara-Yamagami categories $D\TY(A,\chi,\tau)$ (see \cite{Iz,GNN} and Section \ref{4:tycats} below), and quantum group type modular categories
$\CC(\mathfrak{so}_N,q,\ell)$ where $\ell=N$ or $2N$ if $N$ is even or odd respectively (see e.g. \cite{Gan} and Section \ref{2:qgcats} below). The main results of Sections \ref{2:qgcats}, \ref{3:class} and \ref{4:tycats} are summarized in:
\begin{thm}\label{summary}
 Suppose that $\CC$ is a braided \textbf{integral} fusion category and:
\begin{enumerate}
 \item[(i)] all simple objects $X$ are self-dual and $\FPdim(X)\in\{1,2\}$ or
\item[(ii)] $\CC$ is modular with $\FPdim(\CC)\in [1,35]\cup\{pq^2,pq^3\}$, $p\not= q$ primes or
\item[(iii)] $\CC=\CC(\mathfrak{so}_N,q,\ell)$ with $\ell=N$ for $N$ even and $\ell=2N$ for $N$ odd or
\item[(iv)] $\CC=D\TY(A,\chi,\tau)_+$, the trivial component of $D\TY(A,\chi,\tau)$ (under the $\Z/2\Z$-grading)
\end{enumerate}
then $\CC$ has property $\F$.
\end{thm}
Note that in (iii) $\ell/2$ must be a perfect square.

To be conservative, our results provide evidence for a weak form of one direction of Conjecture \ref{mainconj}.  While these results are of interest in the representation theory of finite dimensional Hopf algebras, quantum groups and fusion categories generally, the strong form of the conjecture has some far-reaching connections to quantum computing, complexity theory, low-dimensional topology and condensed matter physics.  The interested reader can find details in the survey articles \cite{Das} and \cite{twoparas}.  Roughly, the connections are as follows.  Any (unitary) modular category provides both $\C$-valued multiplicative link invariants (e.g. the Jones polynomial) and a model for a (theoretical) 2-dimensional physical system (e.g. fractional quantum Hall liquids).  A topological quantum computer would be built upon such a physical system and would (probabilistically) approximate the link invariants in polynomial time.  Now the (finite, infinite) dichotomy of braid group image seems to correspond to similar dichotomies in quantum computing (weak, powerful) and computational complexity of link invariants (easy, hard).  By a ``powerful'' quantum computer we mean \emph{universal} and the corresponding (classical) computational complexity class is $\#P$-hard (where the last dichotomy assumes $P\not= NP$).

\section{The Property $\mathbf{F}$ Conjecture}\label{1:propf}

\begin{defn}
 A braided fusion category $\CC$ has \emph{property} $\mathbf{F}$ if the
associated braid group representations on the centralizer algebras $\End(X^{\ot n})$ have finite image
for all $n$ and all objects $X$.
\end{defn}

Recall that $\dim(\CC)$ is the sum of the squares of the categorical dimensions of (isomorphism classes of) simple
objects.  The Frobenius-Perron dimension (see \cite{ENO}) of a simple object $\FPdim(X)$ is defined to be the largest positive eigenvalue of the fusion matrix of $X$, i.e. the matrix representing $X$ in the left regular representation of the Grothendieck semiring $Gr(\CC)$ of $\CC$.  Similarly, $\FPdim(\CC)$ is the sum of the squares of the Frobenius-Perron dimensions of (isomorphism classes of) simple objects.  We say that the category $\CC$ is \emph{pseudo-unitary} if $\FPdim(\CC)=\dim(\CC)$, which is indeed the case when $\CC$ is unitary (see e.g. \cite{wenzlcstar}).

\begin{defn}
 A fusion category $\CC$ is called \emph{weakly integral} if $\FPdim(\CC)\in\N$, and \emph{integral} if $\FPdim(X)\in\N$ for each simple object $X$.
\end{defn}
It is known  (see e.g. \cite[Proposition 8.27]{ENO}) that $\CC$ is weakly integral if and only if
$\FPdim(X)^2\in\N$ for all simple objects $X$.
We can now state:
\begin{conj}\label{mainconj}
A unitary ribbon category $\CC$ has property $\mathbf{F}$ if, and only if,
$\dim(\CC)\in\N$.  More generally, a braided fusion category has property
$\mathbf{F}$ if, and only if, $\CC$ is weakly integral.
\end{conj}

We note that in a sense property $\F$ is a property of objects: if we denote by $\CC[X]$ the full braided fusion subcategory generated by an object $X$ then it is clear that $\CC$ has property $\F$ if, and only if, $\CC[X]$ has property $\F$ for each object $X$.  We have the following (c.f. \cite[Lemma 2.1]{ERW}):
\begin{lem}\label{genlem}
Let $\mathcal{S} \subset \CC$ be a set of objects such that every simple object of $\CC$ is isomorphic to a subobject of $X^{\ot n}$ for some $X\in\mathcal{S}$ and $n\in\N$. Then $\CC$ has property $\F$ if, and only if, $\CC[X]$ has property $\F$ for each $X\in\mathcal{S}$.
\end{lem}
\begin{proof} The ``only if" direction is clear.  Suppose that $\CC[X]$ has property $\F$ for each $X$ in a generating set, and let $Y$ be a subobject of $X$, with monomorphism $q\in\Hom(Y,X)$.  Since $\CC$ is semisimple, $q$ is split so that we have an epimorphism $p\in\Hom(X,Y)$ with
$pq=Id_Y$ and $(qp)^2=(qp)$.  As the braiding is functorial, we can use (tensor powers of) $p$ and $q$ to construct intertwining maps between $\End(Y^{\ot n})$ and $\End(X^{\ot n})$, and conclude that the braid group image on $\End(Y^{\ot n})$ is a quotient of the braid group image on $\End(X^{\ot n})$.  This shows that if $\CC[X]$ has property $\F$ for each $X$ is a generating set, then $\CC[X_i]$ has property $\F$ for each simple $X_i$.  Similar arguments (restricting to the pure braid group $\mathcal{P}_n$) show that the braid group acts by a finite group on direct sums so that $\CC$ has property $\F$.
\end{proof}

The following definition is not the original formulation of group-theoreticity, but is equivalent by a theorem of \cite{Nat}:
\begin{defn}
A fusion category $\CC$ is \emph{group-theoretical} if its Drinfeld
center $Z(\CC)$ is braided monoidally equivalent to the category of representations of the twisted double $D^\omega G$ of a finite group $G$.
\end{defn}
Group-theoretical categories are integral, but there are many examples of integral non-group-theoretical braided fusion categories (see \cite{Nik}).
Essentially the only general sufficient condition for property $\mathbf{F}$ is the following:

\begin{prop}[\cite{ERW}]\label{erwprop}
Braided group-theoretical categories have property $\mathbf{F}$.
\end{prop}

There are a few other sufficient conditions for an integral fusion category to be group-theoretical available in
the literature.  We collect some of them in:

\begin{prop}\label{dgnoprop1}

Suppose $\CC$ is an integral fusion category.  Then $\CC$ is group-theoretical if:
\begin{enumerate}
 \item $\FPdim(\CC)=p^n$ \cite[Corollary 6.8]{DGNO}
\item $\FPdim(\CC)=pq$ \cite[Theorem 6.3]{EGO}, or
\item $\FPdim(\CC)=pqr$ \cite[Theorem 9.2]{ENO2}
\end{enumerate}
where $p,q$ and $r$ are distinct primes.
\end{prop}

For the next criterion we need two definitions.  For any subcategory $\mathcal{D}\subset\CC$ of a braided fusion category denote by
$\mathcal{D}^\prime$ the \emph{centralizer} of $\mathcal{D}$, i.e. the subcategory consisting of objects
$Y$ for which $c_{X,Y}c_{Y,X}=Id_{X\ot Y}$ for all objects $X$ in $\mathcal{D}$.  By (a generalized version of) a theorem of M\"uger \cite{Mu1} this is equivalent to $\tilde{s}_{X,Y}=\dim(X)\dim(Y)$ for simple $X$ and $Y$ where $\tilde{s}$ is the normalized modular $S$-matrix (see Section \ref{2:qgcats}).  Also, following \cite{ENO} we define $(\mathcal{D})_{ad}$ to be the smallest fusion subcategory of $\CC$ containing $X\ot X^*$ for each simple object $X$ in $\mathcal{D}$.  In \cite{GN}, a fusion category $\mathcal{N}$ is defined to be \emph{nilpotent} if the sequence
$\mathcal{N}\supset\mathcal{N}_{ad}\subset (\mathcal{N}_{ad})_{ad}\supset\cdots$ converges to $Vec$ the fusion category of vector spaces.

Modular group-theoretical categories are characterized by:

\begin{prop}[\cite{DGNO}]\label{dgnoprop2}
A modular category $\CC$ is group theoretical if and only if it is integral
and there is a symmetric subcategory $\mathcal{L}$ such that $(\mathcal{L}^\prime)_{ad}\subset\mathcal{L}$.

\end{prop}

Here a symmetric subcategory $\mathcal{L}$ is one for which $\tilde{s}_{X,Y}=\dim(X)\dim(Y)$ for all simple objects $X$ and $Y$ in $\mathcal{L}$.  In fact, all of the hypotheses of this proposition can be checked once we have determined the $\tilde{s}$-matrix, since one may compute the fusion rules from $\tilde{s}$ to determine $\mathcal{L}_{ad}$.

Group-theoretical categories also have the following useful characterization (see \cite{O2}): a fusion category $\CC$ is group-theoretical if, and only if, the category $\CC_\mathcal{M}^*$ dual to $\CC$ with respect to some indecomposable module category $\mathcal{M}$ is pointed (that is, if $\CC$ is Morita equivalent to a pointed fusion category).  More generally, a fusion category $\CC$ is defined in \cite{ENO2} to be \emph{weakly group-theoretical} if $\CC$ is Morita equivalent to a nilpotent fusion category $\mathcal{N}$.  It follows from \cite{GN} and \cite[Corollary 8.14]{ENO} that any weakly group-theoretical fusion category is weakly integral.  To our knowledge, there are no known examples of weakly integral fusion categories that are not weakly group-theoretical.  This provides further conceptual evidence for the validity of Conjecture \ref{mainconj}. Unfortunately it is not clear how to generalize the proof of Proposition \ref{erwprop} to the weakly group-theoretical setting.

\section{Quantum group type categories}\label{2:qgcats}
Associated to any semisimple finite dimensional Lie algebra $\g$ and a complex number $q$ such that $q^2$ is a primitive $\ell$th root
of unity is a ribbon fusion category $\CC(\g,q,\ell)$.  The construction is essentially due to Andersen (\cite{andersen}) and his collaborators.  We refer the
reader to the survey paper \cite{Survey} and the texts \cite{BK} and \cite{Tur} for a more complete treatment.

Here we will consider two special cases of this construction which yield weakly integral modular categories: $\g=\mathfrak{so}_N$ and with $\ell=2N$ for $N$ odd (type $B$) and $\ell=N$ for $N$ even (type $D$).  In these two cases we will denote $\CC(\so_N,q,\ell)$ by $\CC(B_r)$ and $\CC(D_r)$ for $N=2r+1$ and $N=2r$ respectively with the choice $q=e^{\pi\im/\ell}$.  We remark that in the physics literature these categories are often denoted $SO(N)_2$ corresponding to the tensor category of level $2$ (integrable highest weight) modules over the affine Kac-Moody algebra $\hat{\mathfrak{so}}_N$ equipped with the fusion tensor product (see \cite{fink}).
In both of these cases we find that the simple objects have dimensions in $\{1,2,\sqrt{\ell/2}\}$.  Moreover, the simple objects with dimensions $1$ and $2$ generate ribbon fusion subcategories which we will denote by $\CC(B_r)_0$ and $\CC(D_r)_0$.  Our results can be summarized as follows:

\begin{enumerate}
 \item When $\sqrt{\ell/2}\in\N$ $\CC(B_r)$ and $\CC(D_r)$ have property $\F$ (Theorems \ref{Bintegral} and \ref{Dintegral})
\item In any case $\CC(B_r)_0$ and $\CC(D_r)_0$ have property $\F$ (Theorem \ref{gradedBD}).
\end{enumerate}

\begin{remark}
\begin{enumerate}
 \item[(i)] That the weakly integral categories $\CC(B_1)$ and $\CC(B_2)$ have property $\F$ follows from \cite{jones86,jonescmp}.  The degenerate cases $\CC(D_2)$ and $\CC(D_3)$ can also be shown to have property $\F$ via the identifications $\so_4\cong\ssl_2\times\ssl_2$ (using \cite{jones86}) and $\so_6\cong\ssl_4$ (see \cite[page 192]{FLW}).  It can be shown that $\CC(B_3)$ and $\CC(D_5)$ also have property $\F$ but the computation would take us too far afield, so we leave this for a future paper.  While Conjecture \ref{mainconj} predicts that $\CC(B_r)$ and $\CC(D_r)$ have property $\F$ for any $r$, we do not yet have sufficiently complete information to work these out.
\item[(ii)] Property $\F$ does not depend on the particular choice of a root of unity $q$ since the matrices representing the braid group generators are defined over a Galois extension of $\Q$.
\end{enumerate}

\end{remark}

There are some well-known facts that we will use below, we recall them here along with some standard notational conventions for future reference.  Firstly, the twist coefficient corresponding to a simple object $X_\la$ in $\CC(\g,q,\ell)$ is given by
$$\theta_\la=q^{\lan\la+2\rho,\la\ra}$$ where $\lan\/,\/\ra$ is normalized so that $\lan\al,\la\ra=2$ for short roots and $\rho$ is half the sum of the positive roots.  We will denote by $N_{\la,\mu}^\nu$ the multiplicity of the simple object $X_\nu$ in the tensor product decomposition of $X_\la\ot X_\mu$, and $\tilde{s}$ will denote the normalization of the $S$-matrix with entries $\ts_{\la,\mu}$ with $\ts_{\zero,\zero}=1$.
We also have the following dimension formula:
$$\dim(X_\la)=\prod_{\al\in\Phi_+}\frac{[\lan\la+\rho,\al\ra]}{[\lan\rho,\al\ra]}$$
where $[n]:=\frac{q^n-q^{-n}}{q-q^{-1}}$.
When convenient we will denote by $\nu^*$ the label of $(X_\nu)^*$.  These quantities are related by the useful formula:

\begin{equation}\label{modeq}
 \theta_\la\theta_\mu\ts_{\la,\mu}=\sum_\nu N_{\la^*,\mu}^\nu\theta_\nu\dim(X_\nu)
\end{equation}

\subsection{Type B categories}
Now let us take $\g=\mathfrak{so}_{2r+1}$ and $\ell=4r+2$, with $q=e^{\pi i/\ell}$ for concreteness.  For this choice of $q$ the categories are all unitary (\cite{wenzlcstar}), so that $\dim(X)>0$ for each object $X$ and hence coincides with $\FPdim$.

We use the standard labeling convention for the fundamental weights of type $B$:
$\la_1=(1,0,\ldots,0),\ldots,\la_{r-1}=(1,\ldots,1,0)$ and $\la_r=\frac{1}{2}(1,\ldots,1)$.  Observe that the highest root is $\theta=(1,1,0,\ldots,0)$ and $\rho=\frac{1}{2}(2r-1,2r-3,\ldots,3,1)$.
From this we determine the labeling set for the simple objects in $\CC(B_r)$ and order them as follows:
 $$\{\mathbf{0},2\la_1,\la_1,\ldots,\la_{r-1},2\la_r,\la_r,\la_r+\la_1\}.$$
  For notational convenience we will denote by $\ve=\la_r$ and $\ve^\prime=\la_1+\la_r$.  In addition we adopt the following notation from \cite{Gan}: $\la_i=\ga^i$ for $1\leq i\leq r-1$ and $\ga^r=2\la_r$.   The dimensions of the simple objects are easily computed, we have:
$\dim(X_{\mathbf{0}})=\dim(X_{2\la_1})=1$, $\dim(X_{\ga^i})=\dim(X_{\la_i})=2$ for $1\leq i\leq r$, and
$\dim(X_{\ve})=\dim(X_{\ve^\prime})=\sqrt{2r+1}$.  Thus $\CC(B_r)$ has rank $r+4$ and dimension $4(2r+1)$ and is weakly integral.

Let us denote by $\tilde{s}(\la,\mu)$ the entry of $\tilde{s}$ corresponding to $X_{\la}$ and $X_{\mu}$.  From \cite{Gan} we compute the following:

\begin{eqnarray*}
&& \ts(2\la_1,2\la_1)=1,\quad \ts(2\la_1,\ga^i)=2, \quad \ts(2\la_1,\ve)=\ts(2\la_1,\ve^\prime)=-\sqrt{2r+1}\\
&&\ts(\ga^i,\ga^j)=4\cos(\frac{2ij\pi}{2r+1}),\quad \ts(\ga^i,\ve)=\ts(\ga^i,\ve^\prime)=0\\
&&\ts(\ve,\ve^\prime)=-\ts(\ve,\ve)=\pm\sqrt{2r+1}
\end{eqnarray*}

The remaining entries of $\tilde{s}$ can be determined by the fact that $\tilde{s}$ is symmetric.

One can determine the fusion rules for $\CC(B_r)$ by antisymmetrizing the multiplicities for $\mathfrak{so}_{2r+1}$ with respect to the ``dot action'' of the affine Weyl group, or by the Verlinde formula.  In any case we see that $X_{\ve}$ generates $\CC(B_r)$, with tensor product decomposition rules:
\begin{enumerate}
\item $X_{\ve}\ot X_{\ve}=X_{\mathbf{0}}\oplus\bigoplus_{i=1}^{r} X_{\ga^i}$
\item $X_{\ve}\ot X_{\ga^i}=X_{\ve}\oplus X_{\ve^\prime}$ for $1\leq i\leq r$
\item $X_{\ve}\ot X_{\ve^\prime}=X_{2\la_1}\oplus\bigoplus_{i=1}^{r}X_{\ga^i}$
\item $X_{\ve}\ot X_{2\la_1}=X_{\ve^\prime}$
\end{enumerate}

Moreover we see that $\CC(B_r)$ has a faithful $\Z_2$-grading (see Section \ref{pq^n} below for the definition).  The $0$-graded part $\CC(B_r)_0$ is generated (as an Abelian category) by the simple objects of dimensions $1$ and $2$ while the $1$-graded part $\CC(B_r)_1$ has simple objects $\{X_\ve, X_{\ve^\prime}\}$.

We note that the Bratteli diagram describing the inclusions of the simple components of $\End(X_{\ve}^{\ot n-1})\subset\End(X_{\ve}^{\ot n})$ is precisely the same as the one associated with the Fateev-Zamolodchikov model for $\Z_{2r+1}$ found in \cite{jonessurvey}.

\subsubsection{Type $B$ integral cases}
Observe that $\CC(B_r)$ is integral if, and only if, $2r+1$ is a perfect square.
Let $2r+1=t^2$ for some (odd) integer $t$.  Consider the category $\D(B_r)$
generated by $\one$, $V:=X_{2\la_1}$ and $W_i:=X_{\ga^{it}}$ where $1\leq i\leq (t-1)/2$.
\begin{lem}\label{blemma}
$\D(B_r)$
is symmetric,
and has simple objects $\one$, $V$ and $W_i$ ($1\leq i\leq (t-1)/2$).
\end{lem}
\begin{proof}

We must first verify that the abelian category generated by $\{\one,V,W_i\}$ with $1\leq i\leq (t-1)/2$ is closed under the tensor product.  First observe that since $\FPdim(W_i)=2$ and each object in $\CC(B_r)$ is self-dual, we have $W_i^{\ot 2}=\one\oplus V\oplus X_{\ga^j}$ for some $j$.  We claim that $t\mid j$, so that $X_{\ga^j}=W_{j/t}$.  Indeed, from equation (\ref{modeq}) we have:

$$4=(\theta_{\ga^{it}})^2\ts_{\ga^{it},\ga^{it}}=1+\theta_{2\la_1}+2\theta_{\ga^j}.$$ We compute that $\theta_{2\la_1}=1$ which implies that $\theta_{\ga^j}=e^{-2j^2\pi\im/(2r+1)}=1$ hence $t=\sqrt{2r+1}$ divides $j$.  A similar argument shows that $W_i\ot W_j=W_k\oplus W_{k^\prime}$ for $i<j$, and the remaining fusion rules follow by Frobenius reciprocity.
The symmetry of $\D(B_r)$ is clear from the $\ts$-matrix (notice that $\ts(\ga^{it},\ga^{jt})=4\cos(\frac{2itjt\pi}{t^2})=4$).
\end{proof}

We can now prove:
\begin{thm}\label{Bintegral}
 $\CC(B_r)$ is group-theoretical for $2r+1=t^2$, and hence has property $\F$.
\end{thm}
\begin{proof}

We will verify the hypotheses of Proposition \ref{dgnoprop2}.
Clearly all simple objects have integral dimension and by Lemma \ref{blemma} $\D(B_r)$ is symmetric.  We claim that $(\D(B_r)^\prime)_{ad}\subset\D(B_r)$.  It is enough to show that $\D(B_r)^\prime\subset\D(B_r)$ since $\D(B_r)_{ad}\subset \D(B_r)$.  For this we will demonstrate that if $Z$ is a simple object in $\CC(B_r)$ satisfying $\ts_{Z,W_i}=\dim(Z)\dim(W_i)$ then $Z\in\D(B_r)$.  First notice that $X_\ve$ and $X_{\ve^\prime}$ cannot centralize $W_i$ since the corresponding $\ts$ entry is $0$.  If $X_{\ga^j}$ centralizes $W_1$ we have
$$\ts_{\ga^{t},\ga^j}=4\cos(\frac{2tj\pi}{t^2})=4\cos(\frac{2j\pi}{t})=\dim(W_1)\dim(X_{\ga^j})$$
which implies that $t\mid j$ and so $X_{\ga^j}\in\D(B_r)$.  Thus only objects in $\D(B_r)$ can centralize $W_1$ and so $\D(B_r)^\prime\subset \D(B_r)$ and the hypotheses of Proposition \ref{dgnoprop2} are satisfied.  Hence $\CC(B_r)$ is group-theoretical and hence has property $\mathbf{F}$.
\end{proof}

\subsection{Type D categories}
Now let us take $\g=\mathfrak{so}_{2r}$ and $\ell=2r$, with $q=e^{\pi i/\ell}$.  Observe that $\CC(D_r)$ is unitary so that the function $\dim$ coincides with $\FPdim$.

  The fundamental weights are denoted $\la_1=(1,0,\ldots,0),\ldots\la_{r-2}=(1,\ldots,1,0,0)$, for $1\leq i\leq r-2$
with $\la_{r-1}=\frac{1}{2}(1,\ldots,1,-1)$ and $\la_{r}=\frac{1}{2}(1,\ldots,1)$ the two fundamental spin representations.  We compute the labeling set for $\CC(D_r)$ and order them as follows:
$$
\{\mathbf{0},2\la_1,2\la_{r-1},2\la_r,\la_1,\cdots,\la_{r-2},\la_{r-1}+\la_r,\la_{r-1},\la_r,\la_1+\la_{r-1},\la_1+\la_r\}.$$
For notational convenience we will denote by $\ve_1=\la_{r-1}$, $\ve_2=\la_{r}$, $\ve_3=\la_1+\la_{r-1}$
and $\ve_4=\la_1+\la_r$ and set $\ga^j=\la_j$ for $1\leq j\leq r-2$ and $\ga^{r-1}=\la_{r-1}+\la_r$.  In this notation the dimensions of the simple objects are:
$\dim(X_{\ga^j})=2$ for $1\leq i\leq r-1$, $\dim(X_\mathbf{0})=\dim(X_{2\la_1})=\dim(X_{2\la_{r-1}})=\dim(X_{2\la_{r}})=1$ and $\dim(X_{\ve_i})=\sqrt{r}$ for $1\leq i\leq 4$.  The rank of $\CC(D_r)$ is $r+7$ and $\dim(\CC(D_r))=8r$ so that $\CC(D_r)$ is weakly integral.

The tensor product rules and $\ts$-matrix for $\CC(D_r)$ take different forms depending on
the parity of $r$.  The $\ts$-matrix entries can be recovered from \cite{Gan}, and we list those that are important to our calculations below.  We again denote by $\ts(\la,\mu)$ the $\ts$-entry corresponding to the pair $(X_\la,X_\mu)$:

\begin{eqnarray*}
&&\ts(2\la_1,2\la_1)=\ts(2\la_1,2\la_{r-1})=\ts(2\la_1,2\la_r)=1\\
&&\ts(2\la_1,\ga^j)=2,\quad \ts(2\la_1,\ve_i)=-\sqrt{r}\\
 &&\ts(2\la_{r-1},2\la_r)=\ts(2\la_{r},2\la_r)=(-1)^r\\ &&\ts(2\la_{r-1},\ga^j)=\ts(2\la_{r-1},\ga^j)=2(-1)^j\\
&&\ts(\ga^i,\ga^j)=4\cos(ij\pi/r),\quad \ts(\ga^j,\ve_i)=0
\end{eqnarray*}
In the case that $r=(2k+1)$, one finds that $X_{\ve_1}$ generates $\CC(D_r)$.  All simple objects are self-dual (i.e. $X\cong X^*$) except for $X_{\ve_i}$ $1\leq i\leq 4$, $X_{2\la_{r-1}}$ and $X_{2\la_{r}}$.

In the case that $r=2r$ is even all objects are self-dual and the subcategory generated by $X_{\ve_1}$ has $k+5$ simple objects
labelled by:
$$\{\zero,2\la_1,2\la_{r-1},2\la_{r},\ga^{2},\ga^4,\ldots,\ga^{r-2}, \ve_1,\ve_4\}.$$
The Bratteli diagram for the sequence of inclusions $\End(X_{\ve_1}^{\ot n})\subset\End(X_{\ve_1}^{\ot n})$ is the same as that of the Fateev-Zamolodchikov model for $\Z_{2k}$ found in
\cite{jonessurvey}.  We caution the reader that this subcategory \emph{is not} modular.
Similarly the (non-modular) subcategory generated by $X_{\ve_2}$ has $k+5$ simple objects, and together they generate the full category $\CC(D_r)$.

For any $r>4$ the category $\CC(D_r)$ has a faithful $\Z_2$-grading, where $\CC(D_r)_0$ is generated by the simple objects of dimension $1$ and $2$ and $\CC(D_r)_1$ has simple objects $X_{\ve_i}$, $1\leq i\leq 4$.

\subsubsection{Type $D$ integral cases}

Observe that if $r=2^{2t}$ then the dimension of each object in $\CC(D_r)$ is an integer since $\sqrt{2^{2t}}=2^t$.  Moreover, $8r$ is a power of $2$ so that Propositions \ref{erwprop} and \ref{dgnoprop1} immediately imply that $\CC(D_r)$ has property $\F$ in this special case.

More generally, we will show that when $r=x^2$ is a perfect square
the category $\CC(D_r)$ is group theoretical.  Denote $V:=X_{2\la_1}$, $U:=X_{2\la_{r-1}}$, $U^\prime=X_{2\la_r}$ and $Z_i:=X_{\ga^{2xi}}$ with $i\leq (x^2-2)/2x$ (note that for $r=4$ there are no $Z_i$).  For $r$ even,
define $\D_e(D_r)$ be the subcategory generated by $Z_i$, $V$, $U$ and $U^\prime$.  For $r$ odd define $D_o(\D_r)$ to be the subcategory generated by $W_i$ and $V$.
\begin{lem}
The subcategories $\D_e(D_r)$ and $\D_o(D_r)$ are symmetric and the sets $\{\one,V,Z_i\}$ (resp. $\{\one,V,U,U^\prime,Z_i\}$) are all simple objects in $\D_o(D_r)$ (resp. $\D_e(D_r)$).
\end{lem}
\begin{proof}
  As in the type $B$ case we verify that the sets given represent all simple objects by exploiting the equation (\ref{modeq}).  For example to see that $Z_i\ot Z_j$ contains only the simple objects listed above, we compute that $\theta_{\ga^j}=q^{j(2x^2-j)}=1$ if, and only if, $2x\mid j$ for $q=e^{\pi \im/2x^2}$, and $\theta_{2\la_r}=\theta_{2\la_{r-1}}=(\im)^r$.  Thus the fact that $\ts(Z_i,Z_j)=4$ implies that any simple subobject $X$ of $Z_i\ot Z_j$ must have $\theta_X=1$ which is sufficient to conclude that such an $X$ is as we have listed.
It is immediate from the $\ts$-matrix entries listed above that the given categories are symmetric since the condition $\ts_{i,j}=\dim(X_i)\dim(X_j)$ is satisfied by all pairs of objects.
\end{proof}

We can now prove:
\begin{thm}\label{Dintegral}
 $\CC(D_r)$ is group-theoretical for $r=x^2$, and hence has property $\F$.
\end{thm}
\begin{proof}
We need only verify that $(\D_o(D_r)^\prime)_{ad}\subset \D_o(D_r)$ and $(\D_e(D_r)^\prime)_{ad}\subset \D_e(D_r)$.  In the case $r=x^2$ is even it is clear from the $\ts$-matrix entries listed above that $\D_e(D_r)^\prime=\D_e(D_r)$ since no $X_{\ve_i}$ centralizes $V$ and $Z_1$ is not centralized by any $X_{\ga^j}$ with $2x\nmid j$.  Since $\D_e(D_r)$ is a tensor-subcategory the result follows from Proposition \ref{dgnoprop2} (for $r\geq 6$, the case $r=4$ follows from Proposition \ref{dgnoprop1}).

For $r$ odd we see that $U$ and $U^\prime=U^*$ are in $\D_o(D_r)^\prime$ but not in $\D_o(D_r)$.  However, $U\ot U^*=U\ot U^\prime\cong\one$ so that we still have $(\D_o(D_r)^\prime)_{ad}\subset \D_o(D_r)$, and the claim follows by Proposition \ref{dgnoprop2}.
\end{proof}

\section{Some Classification Results}\label{3:class}

In this section we classify fusion categories whose simple objects have dimensions $1$ or $2$ that are generated by a self-dual object of dimension $2$, as well as integral modular categories of dimension $pq^2$ or $pq^3$.  In all cases we conclude that the categories must be group-theoretical.  These results will be useful later to verify Conjecture \ref{mainconj} in several cases.
\subsection{Dimension $2$ generators}

The following definition was introduced in \cite{RJPAA}:
\begin{defn}
Two fusion categories $\CC$ and $\D$ are \emph{Grothendieck equivalent} if they share the same fusion rules, i.e. $Gr(\CC)$ and $Gr(\D)$ are isomorphic as unital based rings.
\end{defn}

\begin{thm}\label{main}
 Suppose that $\CC$ is a fusion category such that:
\begin{enumerate}
\item $\FPdim(X)\in\{1,2\}$ for any simple object $X$.
\item All objects are \emph{self-dual}, i.e. $X\cong X^*$ (non-canonically isomorphic) for every object $X$.
    \item $\CC=\CC[X_1]$ with $X_1$ simple and $\FPdim(X_1)=2$ (\emph{i.e.} every simple object $Y$ is a subobject of $X_1^{\ot n}$ for some $n$).
        \item $Gr(\CC)$ is commutative.
\end{enumerate}
Then we have:
\begin{enumerate}
 \item[(i)] $\CC$ is Grothendieck equivalent to $\Rep(D_n)$, the representation category of the dihedral group of order $2n$.
\item[(ii)] $\CC$ is group-theoretical.
\end{enumerate}
\end{thm}

The following is immediate:
\begin{cor}\label{maincor}
 Suppose that $\CC$ is a braided fusion category satisfying conditions (1) and (2) of Theorem \ref{main}.  Then $\CC$ has property $\mathbf{F}$.
\end{cor}
\begin{proof}
Every non-pointed simply generated subcategory of $\CC$ satisfies all four conditions of Theorem \ref{main}, so the claim follows from Proposition \ref{erwprop} and Lemma \ref{genlem}.
\end{proof}

\begin{proof} (of Theorem \ref{main}).
 Let $X_1$ be a simple object generating $\CC$.

First suppose that $X_1^{\ot 2}\cong \one \oplus Z_2\oplus Z_3\oplus Z_4$ where $\FPdim(Z_i)=1$.  Then $X_1^{\ot 3}\cong X_1^{\oplus 4}$ since each $Z_i$ is self-dual.  Moreover the $Z_i$ are distinct since $\dim\Hom(X_1\ot X_1,Z_i)=\dim\Hom(X_1\ot Z_i,X_1)=1$ by comparing FP-dimensions.  This implies that $\CC$ is Grothendieck equivalent to $\Rep(D_4)$ and $\FPdim(\CC)=8$ so that $\CC$ is group-theoretical by Proposition \ref{dgnoprop1} above.

Now suppose that $X_1^{\ot 2}\cong \one \oplus Z_2\oplus X_2$ where $\FPdim(X_2)=2$ and $\FPdim(Z_2)=1$.  This implies that $Z_2\ot X_1\cong X_1$, but we must analyze cases for $X_1\ot X_2$.  If $X_1\cong X_2$ we find that $\CC$ is Grothendieck equivalent to $\Rep(D_3)$ by inspection.  If $X_1\not\cong X_2$ then we have three possibilities:

$X_1\ot X_2\cong X_1\oplus \begin{cases}
                   X_3 & \FPdim(X_3)=2, X_3\not\cong X_2\\
Z_3\oplus Z_4 & \FPdim(Z_i)=1\\
X_2 &\\
                 \end{cases}$

In the latter two cases all simple objects appear in $X_1^{\ot 3}$ and all  fusion rules are completely determined: we obtain Grothendieck equivalences with $\Rep(D_6)$ and $\Rep(D_5)$ respectively.   In the first case we proceed inductively.  Assuming that $X_1\ot X_{k-1}\cong X_{k-2}\oplus X_{k}$ where $j$ is minimal such that $X_j$ appears in $X_1^{\ot j}$ and $\FPdim(X_i)=2$ we find that there are three distinct possibilities for $X_1\ot X_k$:
\begin{enumerate}
\item[(a)] $X_{k-1}\oplus X_{k+1}$,
\item[(b)] $X_{k-1}\oplus Z_3\oplus Z_4$ with $\FPdim(Z_i)=1$, or
\item[(c)] $X_{k-1}\oplus X_{k}$.
\end{enumerate}
The finite rank of $\CC$ implies that case (a) cannot be true for all $k$, so that there is some minimal $k$ for which case (b) or (c) holds.  In cases (b) and (c) all fusion rules involving $X_1$ are completely determined, i.e. every simple object appears in $X_1^{\ot n}$ for some $n\leq k+1$.  Moreover, it can be shown that in fact \emph{all} fusion rules are determined in these cases.  We sketch the argument in case (b), case (c) is similar.

Let $k$ be minimal such that $X_1\ot X_{k}\cong X_{k-1}\oplus Z_3\oplus Z_4$ with $\FPdim(Z_i)=1$.  The simple object of $\CC$ are then $\{\one,Z_2,Z_3,Z_4,X_1,\ldots,X_k\}$ where $\FPdim(X_i)=2$ and $\FPdim(Z_i)=1$.
The fusion rules involving $X_1$ are:
 $$X_1\ot X_i\cong X_{i-1}\oplus X_{i+1} \text{ for } i\leq k-1,$$
 $$X_1\ot X_k\cong X_{k-1}\oplus Z_3\oplus Z_4, \; X_1\ot Z_2\cong X_1, \text{ and }
 X_1\ot Z_3\cong X_1\ot Z_4\cong X_{k}.$$
Thus the fusion matrix $N_{X_1}$ is known.  Next we determine the fusion rules involving $Z_3$, (the rules for $Z_4$ essentially the same).  Firstly, $\FPdim(Z_3\ot Z_2)=1$ so $Z_3\ot Z_2\cong Z_4$.  Next we see that $Z_3\ot X_i\cong X_{k-i+1}$.  For $i=1,k$ this is clear, and the rest follows by induction.  From this it follows that $Z_2\ot X_i\cong X_i$ since $Z_2\cong Z_3\ot Z_4$.  Now we use the fact that $X\rightarrow N_X$ is a representation of the Grothendieck semiring of $\CC$ to determine the $N_{X_i}$ for $i>1$ inductively from the fusion rules: $ X_i\cong X_1\ot X_{i-1}\ominus X_{i-2}$ (formally).

Observe that in case (b) $\FPdim(\CC)=4k+4$ and in case (c) $\FPdim(\CC)=4k+2$.  By inspection, we have proved  $\CC$ is Grothendieck equivalent to $\Rep(D_{2k+2})$ or $\Rep(D_{2k+1})$ in cases (b) and (c) respectively. Thus (i) is proved.

Now we proceed to the proof of (ii).  To prove that $\CC$ is group-theoretical we will exhibit an indecomposable module category $\mathcal{M}$ over $\CC$ so that $\CC_{\mathcal{M}}^*$ is a pointed category.  To do this we will produce an algebra $A$ in $\CC$ so that the category $A-\mathrm{bimod}=\CC_{\Rep(A)}^*$ of $A$-bimodules in $\CC$ is pointed
($\Rep(A)$ denotes the category of right $A$-modules in $\CC$).  We follow the method of proof of \cite[Theorem 6.3]{EGO}. We will focus on case (b), as the proof of case (c) is precisely the same.
In case (b) (and (c)) we take $A=\one\oplus Z_2$ as an object of $\CC$.  As in \cite[Page 3050]{EGO}, $Z_2\ot X_1\cong X_1$ implies that $A$ has a unique structure of a semisimple algebra in $\CC$, which is clearly indecomposable (see \cite[Definition 3.2]{O1}).  Thus $\CC_{\Rep(A)}^*$ is a fusion category (see \cite[Theorem 2.15]{ENO}), with unit object $A$.

Notice that $X_i\ot Z_2\cong X_i$ so that $X_i\ot A\cong 2X_i$ as objects of $\CC$.  Thus $X_i$ has two simple (right) $A$-module structures.  Moreover, for any simple $A$-module $M$ with $\Hom(M,X_i)\not=0$ we have $\Hom_A(X_i\ot A,M)\not=0$, so any such $A$-module $M$ is isomorphic to $X_i$.  Fix such an $M$.  From \cite[Example 3.19]{EO} and \cite[Lemma 6.1]{EGO} we see that the internal-$\Hom$ $\uHom(M,M)$
\begin{enumerate}
\item is a subobject of $X_i\ot X_i$,
\item is an algebra and
\item has $\FPdim(\uHom(M,M))=2$.
\end{enumerate}
Since $X_i^{\ot 2}\cong \one\oplus Z_2\oplus X_j$ for $i\not=\frac{k+1}{2}$ (always true if $k$ is even), we find that in these cases $\uHom(M,M)=A$.  Thus, if $i\not=\frac{k+1}{2}$, $Z_2\ot M=M$ (as $A$-modules), and the proof of \cite[Lemma 6.2]{EGO} goes through, showing that each $X_i$, $i\not=\frac{k+1}{2}$, has $4$ $A$-bimodule structures $M_i^{(j)}$, $1\leq j\leq 4$ and each $M_i^{(j)}$ is invertible in $A-\mathrm{bimod}$.  Now consider $X^\prime:=X_{\frac{k+1}{2}}$ ($k$ even).  Let $N_1$ and $N_2$ be the two simple $A$-modules with $N_i=X^\prime$ as objects.  There are two possibilities: $Z_2\ot N_1=N_1$ or $Z_2\ot N_1=N_2$.  In the first case we obtain $4$ invertible $A$-bimodules just as in the other cases.  In the second, we may assume that $\uHom(N_i,N_i)=\one\oplus Z_3$, as $X^\prime\ot X^\prime=\one\oplus Z_2\oplus Z_3\oplus Z_4$.  In this case $L:=N_1\oplus N_2$ has the structure of a simple $A$-bimodule.  Moreover, since $\FPdim(L)$ is integral and $\FPdim(\CC)=4k+4=\FPdim(A-\mathrm{bimod})$ we conclude that $L$ is the unique simple $A$-bimodule with $\FPdim(L)=2$.  But this implies that $M_i^{(j)}\ot L\cong L$ for every $i,j$ since $M_i^{(j)}$ is invertible, a contradiction.  By dimension considerations there are $4$ more simple invertible objects in $A-\mathrm{bimod}$ isomorphic to $\one\oplus Z_2$ or $Z_3\oplus Z_4$, as objects of $\CC$.  Indeed we can identify them: $A$=unit object, $A^\prime$, the kernel of the multiplication map (as an $A$-bimodule morphism) $A\ot A\rightarrow A$.  Fix any $A$-module $T$ with $T=Z_3\oplus Z_4$, as objects of $\CC$, then $\uHom(T,T)=A$ so $T$ has an $A$-bimodule structure $T_1$ and $T_1\ot A^\prime\not=T_1$ is the final invertible object.  Hence $A-\mathrm{bimod}$ is pointed, and (ii) is proved.

\end{proof}

We would like to point out that Theorem \ref{main}(i) is related to some results in other contexts.  In \cite[Corollary 4.6.7(a)]{GHJ} a ``unitary'' version is obtained: it is shown that a pair of $II_1$ subfactors $N\subset M$ of finite depth with (Jones) index $[M:N]=4$, then the principal graph of the inclusion must be the Coxeter graph $D_n^{(1)}$ provided the Perron-Frobenius eigenvector is restricted to have entries $\leq 2$.  See \cite{wenzlcstar} for the connection between unitary fusion categories and $II_1$ subfactors.
More recently in \cite[Theorem 1.1(ii)]{BN} a Hopf algebra version is proved, classifying subalgebras generated by subcoalgebras of dimension $4$ in terms of polyhedral groups.  Our results are for fusion categories, and none of the three versions imply each other.

\begin{remark}\label{nsdrmk}
 We can weaken the hypothesis of Theorem \ref{main} in the following way: remove (2), but insist that the generating object $X_1$ must be self-dual.  Then $\CC$ is still group-theoretical.  We may determine the possible fusion rules in much the same way as above.  First suppose that $X_1^{\ot 2}\cong\one\oplus Z_2\oplus Z_3\oplus Z_4$ with, say $Z_3$ non-self-dual.  Then $X_1$ self-dual implies $Z_3^*\cong Z_4$ and $Z_2^*\cong Z_2$ without loss of generality.  We then see that $X_1\ot Z_i\cong X_1$ exploiting the symmetries of the fusion coefficients $1=N_{X_1,X_1}^{Z_i}=N_{X_1,Z_i^*}^{X_1}$.  Thus in this case $\FPdim(\CC)=8$ and group-theoreticity follows (however, such a fusion category cannot be braided, see \cite{Sie}).  Next suppose that $X_1^{\ot 2}\cong \one\op Z_2\op X_2$ with $\FPdim(X_2)=2$.  Then $X_2$ must be self-dual.  As in the proof of Theorem \ref{main}, we have a minimal $k$ such that $X_1\ot X_i\cong X_{i-1}\op X_{i+1}$ for $i<k$ and either $X_1\ot X_k\cong X_{k-1}\op X_k$ or $X_1\ot X_k\cong X_{k-1}\op Z_3\op Z_4$ where $\FPdim(X_j)=2$ for all $j$ and $\FPdim(Z_i)=1$ for all $i$.  Observe that in either case each $X_k$ is self-dual (by induction).  So the only non-self-dual possibility is that $Z_3^*\cong Z_4$.  As in the proof of Theorem \ref{main}, this determines all fusion rules, and we see that $Gr(\CC)\cong Gr(\Rep(\Z_{k+1}\rtimes\Z_4))$ where the conjugation action of $\Z_4$ is by inversion.  By defining $A:=\one\oplus Z_2$ (and noting that $\one\oplus Z_3$ \emph{is not} an algebra) similar arguments as in proof of Theorem \ref{main}(ii) show that $\CC$ is group-theoretical, which we record in the following:

\begin{lem}\label{smalllemma}
Suppose $\CC$ is Grothendieck equivalent to $\Rep(\Z_k\rtimes\Z_4)$ where conjugation by the generator of $\Z_4$ acts by inversion on $\Z_k$.  Then $\CC$ is group-theoretical.
\end{lem}
\end{remark}


We would like to point out that Theorem \ref{main} implies that any fusion category $\CC$ that is Grothendieck equivalent to $\Rep(D_k)$ is group theoretical.  Let us denote by $\GT$ the class of finite groups $G$ for which any fusion category $\CC$ in the Grothendieck equivalence class $\lan \Rep(G)\ra$ of $\Rep(G)$ is group-theoretical.
\begin{question}\label{grequivconj}
For which finite groups $G$ is it true that if $\CC$ is a fusion category that is Grothendieck equivalent to $\Rep(G)$ then $\CC$ is group-theoretical, i.e. which finite groups are in $\GT$?
\end{question}

It is certainly not the case that group-theoreticity is invariant under Grothendieck equivalence: \cite{GNN} contains an example of a non-group-theoretical category that is Grothendieck equivalent to the group-theoretical category $\Rep(D(S_3))$ (the representation category of the double of the symmetric group $S_3$).
However, it is possible that this holds for all finite groups $G$.
One can often use the technique of proof of Theorem \ref{main}(ii) to verify that a given group $G$ is in $\GT$.

The following gives some (scant) evidence that perhaps $\GT$ contains \emph{all} finite groups:
\begin{prop}
The following groups are in $\GT$:
\begin{enumerate}
\item $D_k$ (Theorem \ref{main})
 \item Any abelian group $A$
\item Any group $G$ with $|G|\in\{p^n,pq,pqr\}$ where $p,q$ and $r$ are distinct primes (Proposition \ref{dgnoprop1})
\item $G\times H$ for $G,H\in\GT$
\item all nilpotent groups (from the previous two)
\item $A_5$ (\cite[Theorem 9.2]{ENO2})
\item $\Z_{p^n}^\times\ltimes\Z_{p^n}$ $p$ prime (\cite[Corollary 7.4]{EGO})
\end{enumerate}
\end{prop}

We have the following application of Corollary \ref{maincor} and Lemma \ref{smalllemma}:

\begin{thm}\label{gradedBD}
For any $r$ the $0$-graded subcategories $\CC(B_r)_0$ and $\CC(D_r)_0$ are group-theoretical and hence have property $\F$.
\end{thm}
\begin{proof}
In the cases $\CC(B_r)_0$ and $\CC(D_r)_0$ with $r$ even the hypotheses of Corollary \ref{maincor} are satisfied since all objects are self-dual.  In the case $r$ is odd, one finds that $\CC(D_r)_0$ is Grothendieck equivalent to $\Rep(\Z_r\rtimes\Z_4)$ as in Lemma \ref{smalllemma} and the claim follows.
\end{proof}

\begin{remark}
 In contrast with group-theoreticity, having property $\F$ seems only to depend on the fusion rules of the category, not the deeper structures (such as specific braiding!).  We ask the following:
\begin{question}
 Is property $\F$ invariant under Grothendieck equivalence?
\end{question}
The truth of Conjecture \ref{mainconj} would answer this in the affirmative since integrality of a braided fusion category \emph{is} invariant under Grothendieck equivalence.  Moreover, if the answer is ``yes'' verifying property $\F$ would be made significantly easier.
\end{remark}

\subsection{FP-dimensions $pq^2$ and $pq^3$}\label{pq^n}

This subsection is partially a consequence of discussions with Dmitri Nikshych, to whom
we are very thankful.

The goal of this subsection is to show that any integral modular category of
dimension less than $36$ is group-theoretical, and hence has property $\F$.
We will need the following two propositions.

First recall that a fusion category is said to be {\em pointed} if all its simple
objects are invertible. For a fusion category $\CC$, we denote the full fusion
subcategory generated by the invertible objects by $\CC_{pt}$.

\begin{prop}
\label{pq^2}
Let $p$ and $q$ be distinct primes. Let $\CC$ be an integral modular category of
dimension $pq^2$. Then $\CC$ must be pointed (in particular group-theoretical).
\end{prop}
\begin{proof}
Suppose $\CC$ is not pointed. We will show that this leads to a contradiction.
By \cite[Lemma 1.2]{EG} (see also \cite[Proposition 3.3]{ENO}),
the possible dimensions of simple objects of $\CC$ are $1$ and $q$.
Let $l$ and $m$ denote the number of $1$-dimensional and $q$-dimensional
objects, respectively, of $\CC$. By dimension count we must have $l + m q^2 = pq^2$,
this forces $l = q^2$, so $\dim (\CC_{pt}) = q^2$.
By \cite[Theorem 3.2 (ii)]{Mu1}, $\dim ((\CC_{pt})') = p$, so $(\CC_{pt})'$ must
be pointed \cite[Corollary 8.30]{ENO}. Therefore, $(\CC_{pt})' \subset \CC_{pt}$,
which implies that $p$ divides $q^2$, a contradiction.
\end{proof}

Recall that a {\em grading} of a fusion category $\CC$
by a finite group $G$ is a decomposition
\[
\CC =\bigoplus_{g\in G}\, \CC_g
\]
of $\CC$ into a direct sum of full Abelian subcategories such that
$\otimes$ maps $\CC_g\times \CC_h$ to $\CC_{gh}$ for all $g,h\in G$.
The $\CC_g$'s will be called {\em components} of the $G$-grading of $\CC$.
A grading is said to {\em faithful} if $\CC_g\neq 0$ for all $g\in G$.
In the case of faithful grading, the FP-dimensions of the components
of the $G$-grading of $\CC$ are equal \cite[Proposition 8.20]{ENO}.

It was shown in \cite{GN} that every fusion category $\CC$ is faithfully graded by a
certain group called {\em universal grading group}, denoted $U(\CC)$.
The $U(\CC)-$grading $\displaystyle \CC = \bigoplus_{x \in U(\CC)} \CC_x$ is
called the {\em universal grading} of $\CC$. For a modular
category $\CC$, the universal grading group $U(\CC)$ of $\CC$
is isomorphic to the group of isomorphism
classes of invertible objects of $\CC$ \cite[Theorem 6.3]{GN}.

\begin{prop}
\label{pq^3}
Let $p$ and $q$ be distinct primes. Let $\CC$ be an integral modular category of
dimension $pq^3$. Then $\CC$ must be pointed (in particular group-theoretical).
\end{prop}
\begin{proof}
Suppose $\CC$ is not pointed. We will show that this leads to a contradiction.
By \cite[Lemma 1.2]{EG} (see also \cite[Proposition 3.3]{ENO}),
the possible dimensions of simple objects of $\CC$ are $1$ and $q$.
By numerical considerations, there are three possible values for $\dim \CC_{pt}$:
$q^3,pq^2$, or $q^2$.

Case (i): $\dim \CC_{pt} = q^3$.
By \cite[Theorem 3.2 (ii)]{Mu1}, $\dim ((\CC_{pt})') = p$, so $(\CC_{pt})'$ must
be pointed \cite[Corollary 8.30]{ENO}. Therefore, $(\CC_{pt})' \subset \CC_{pt}$,
which implies that $p$ divides $q^3$, a contradiction.

Case (ii): $\dim \CC_{pt} = pq^2$.
In this case, the components of the universal grading of $\CC$ have dimensions
equal to $q$, so they can not accommodate an object of dimension $q$,
a contradiction.

Case (iii): $\dim \CC_{pt} = q^2$.
In this case, the components of the universal grading of $\CC$ have dimensions
equal to $pq$. By dimension count, each component must contain
at least $q$ invertible objects. Since there are $q^2$ components the previous
sentence implies that $\CC$ contains at least $q^3$ invertible objects,
a contradiction.
\end{proof}

Propositions \ref{dgnoprop1}, \ref{pq^2}, and \ref{pq^3} establish the following:

\begin{prop}
\label{dim<36 implies g-t}
Any integral modular category of dimension less than 36 is group-theoretical,
and hence has property $\F$.
\end{prop}

\begin{example}\label{su3ex}
The following example illustrates: 1) that for integral braided fusion categories group-theoreticity is not necessary for property $\F$, 2) that hypotheses (3) and (4) of Theorem \ref{main} are not sufficient to conclude group-theoreticity and 3) that the assumption $\FPdim(\CC)<36$ of Proposition \ref{dim<36 implies g-t} is necessary.

Let $\CC=\CC(\mathfrak{sl}_3,e^{\pi\im/6},6)$ (in the notation of Section \ref{2:qgcats}).   This category has rank $10$ and $\dim(\CC)=36$.  We order the simple objects
$\one,X_3,X_3^*,Y,X_1,X_1^*,X_2,X_2^*,Z$ and $Z^*$, where $\dim(X_3)=1$, $\dim(X_1)=\dim(X_2)=\dim(Z)=2$ and $\dim(Y)=3$.  The $S$-matrix is of the form:
$\begin{pmatrix}
   A & B\\
B^t & C
  \end{pmatrix}$
where $$A=\begin{pmatrix}
          1 & 1 & 1& 3\\
1 & 1 & 1& 3\\
1 & 1 & 1& 3\\
3 & 3 & 3& -3\\
         \end{pmatrix},
B=2\begin{pmatrix}
    1&1&1&1&1&1\\
\om&1/\om& 1/\om& \om& \om& 1/\om\\
1/\om&\om& \om& 1/\om& 1/\om& \om\\
0&0&0&0&0&0
   \end{pmatrix}.$$ Here $\om=e^{2\pi\im/3}$
and $C_{i,j}=2\zeta^k$ where $\zeta=e^{\pi\im/9}$ and $\pm k\in\{1,5,7\}$.
The corresponding twists are: $$(1,1,1,-1,\zeta^4,\zeta^4,\zeta^{10},\zeta^{10},\zeta^{16},\zeta^{16}).$$

We claim that $\CC$ is not group-theoretical.
There are two tensor subcategories.  The first, $\D$, generated by $X_3$ has rank $3$ and the other is the centralizer $\D^\prime$ of $\D$ generated by $Y$.  The important fusion rules are $Y^{\ot 2}=\one\oplus X_3\oplus X_3^*\oplus 2Y$, and $X_3^{\ot 2}=X_3^*$.  We can see from the $S$-matrix that $\D$ is the only non-trivial symmetric subcategory.  Moreover, $(\D^\prime)_{ad}\not\subset\D$ since $Y\in\D^\prime$ is a subobject of $Y^{\ot 2}$ which is not in $\D$, so by Proposition \ref{dgnoprop2}, $\CC$ is not group theoretical.

This category is known to have property $\F$; we were made aware of this by Michael Larsen \cite{Lpc}.

\end{example}

\section{Applications to Doubled Tambara-Yamagami Categories}\label{4:tycats}

In \cite{TY} D.~Tambara and S.~Yamagami completely classified
fusion categories satisfying certain fusion rules in which all but one
simple object is invertible.
They showed that such categories are parameterized by triples
$(A,\chi, \tau)$, where $A$ is a
finite abelian group, $\chi$ is a nondegenerate symmetric bilinear form
on $A$, and $\tau$ is square root of $|A|^{-1}$.
We will denote the category associated to any such triple
by $\TY(A,\chi, \tau)$.
The category $\TY(A,\chi, \tau)$ is described as
follows. It is a skeletal category with simple objects $\{a\mid a\in
A\}$ and $m$, and tensor product
\[
a\ot b =ab,\quad a\ot m = m,\quad m \ot a=m,\quad m \ot m
=\bigoplus_{a\in A}\, a,
\]
for all $a, b\in A$ and the unit object $e\in A$.
The associativity constraints are defined via $\chi$.
The unit constraints are the identity maps.
The category $\TY(A,\chi,\tau)$ is rigid with $a^*=a^{-1}$ and $m^*=m$
(with obvious evaluation and coevaluation maps). It
has a canonical spherical structure with respect to
which categorical and Frobenius-Perron dimensions
coincide (i.e., $\TY(A,\chi,\tau)$ is pseudo-unitary). Therefore,
the Drinfeld center $D\TY(A,\chi,\tau)$ of $\TY(A,\chi,\tau)$
is a (pseudo-unitary)
modular category.
The following parameterization of simple objects of
$D\TY(A,\chi,\tau)$ can be deduced from \cite{Iz}:

\begin{prop}
\label{simples in DTY}
Simple objects of $D\TY(A,\chi,\tau)$ are parameterized as
follows:
\begin{enumerate}
\item[(1)] $2|A|$ invertible objects $X_{a, \delta}$, where $a\in A$ and
$\delta$ is a square root of $\chi(a,a)^{-1}$.
Also, $X^*_{a, \delta} = X_{a^{-1}, \delta}$;
\item[(2)] $\frac{|A|(|A|-1)}{2}$ two-dimensional objects $Y_{a,b}$,
where $(a,\,b)$ is an unordered pair of distinct objects in $A$.
Also, $Y^*_{a, b} = Y_{a^{-1}, b^{-1}}$;
\item[(3)] $2|A|$ objects  $Z_{\rho, \Delta}$ of dimension $\sqrt{|A|}$,
 where $\rho$ is a linear $\chi$-character of $A$ and
$\Delta$ is a square root of $\tau \sum_{x\in A}\, \rho(x)$.
\end{enumerate}
\end{prop}

We will use the following fusion rules \cite{Iz} in the sequel:

\begin{lem}
\label{fusion rules in DTY}
Set $Y_{a,a} := X_{a,\delta} \oplus X_{a,-\delta}$. Then
\begin{enumerate}
\item $X_{a, \delta} \ot X_{a', \delta'} = X_{aa',
\delta\delta'\chi(a,a')^{-1}}$.

\item $X_{a, \delta} \ot Y_{b,c} = Y_{ab,ac}$.

\item $Y_{a,b} \ot Y_{c,d} = Y_{ac,bd} \oplus Y_{ad,bc}$.

\end{enumerate}
\end{lem}

Note that $D\TY(A,\chi,\tau)$ admits a $\mathbb{Z}/2\mathbb{Z}$-grading:
$$
D\TY(A,\chi,\tau) = D\TY(A,\chi,\tau)_+ \oplus D\TY(A,\chi,\tau)_-,
$$
where $D\TY(A,\chi,\tau)_+$ is the full fusion subcategory generated by
objects $\{ X_{a,\delta},Y_{b,c}\}$ and $D\TY(A,\chi,\tau)_-$ is the
full abelian subcategory generated by objects $\{Z_{\rho, \Delta}\}$.

\begin{prop}
\label{DTY_+ is g-t}
The trivial component $D\TY(A,\chi,\tau)_+$ of $D\TY(A,\chi,\tau)$ (under the
$\mathbb{Z}/2\mathbb{Z}$-grading) is group-theoretical, and hence has property
$\mathbf{F}$.
\end{prop}
\begin{proof}
The proof is similar to that of Theorem \ref{main}(ii).  We take the algebra $A=\one\oplus X$ where $X:=X_{e,-1}$.  By computing $\uHom$, each simple object of the form $Y_{a,b}$ corresponds to $4$ invertible $A$-bimodules unless $a^2=b^2$, in which case:
$$Y_{a,b}\ot Y_{a,b}^*=\one\oplus X\oplus X_{ab^{-1},\delta}\oplus X_{ab^{-1},-\delta}.$$  Let $M_1,M_2$ be the simple $A$-modules with $M_i=Y_{a,b}$ and $a^2=b^2$, and suppose that $X\ot M_1=M_2$.  Then $L=M_1\oplus M_2$ is a simple $A$-bimodule with $\FPdim(L)\geq 2$.  Let $N$ be an invertible $A$-bimodule with $N=Y_{c,d}$ for some $c,d$ with $c^2\not=d^2$.  Then $N\ot L$ is a subobject of $2(Y_{a,b}\ot Y_{c,d})=2(Y_{ac,bd}\oplus Y_{ad,bc})$.  But $(ad)^2\not=(bc)^2$ and $(ac)^2\not=(bd)^2$ as $c^2\not=d^2$, so all sub-bimodules of $N\ot L$ are invertible and in particular $N\ot L$ is not simple.  This is a contradiction, so we must have $\uHom(M_i,M_i)=A$, and $Y_{a,b}$ corresponds to $4$ invertible $A$-bimodules in all cases.

Finally we observe that each $X_{a,\delta}\oplus X_{a,-\delta}$ has two $A$-bimodule structures, each of which is invertible.  Thus the dual to $D\TY(A,\chi,\tau)_+$ with respect to $\Rep(A)$ is pointed, and the proposition is proved.

\end{proof}

\begin{remark}
Here is another proof of Proposition \ref{DTY_+ is g-t}: In \cite[Section 4]{GNN},
it was shown that $D\TY := D\TY(A,\chi,\tau)$ is equivalent to a
$\mathbb{Z}/2\mathbb{Z}$-equivariantization of a certain fusion category
$\E$ (which we describe below), i.e., $D\TY \cong \E^{\mathbb{Z}/2\mathbb{Z}}$. It follows from
the arguments in \cite[Section 4]{GNN} that the trivial component $D\TY_+$
is equivalent to the $\mathbb{Z}/2\mathbb{Z}$-equivariantization of
the pointed part of $\E$, i.e., $D\TY_+ \cong (\E_{pt})^{\mathbb{Z}/2\mathbb{Z}}$.
It follows from \cite[Theorem 3.5]{Nik} that equivariantizations of pointed categories
are group-theoretical; therefore, $D\TY_+$ is group-theoretical.
Let us describe the aforementioned fusion category $\E$ specifically:
Let $\TY = \TY(A,\chi,\tau)$ and let $\TY_{pt}$ denote the pointed
part of $\TY$. Then $\E = Z_{\TY_{pt}}(\TY)$, the {\em relative center}
(see \cite[Subsection 2.2]{GNN}) of $\TY$. Note that $\E$ is a
braided $\Z/2\Z$-crossed fusion category in the sense of \cite{T}.
\end{remark}

\begin{remark}
\label{C gt iff D(C) gt}
Let $\CC$ be a fusion category. It is well known that $\CC$ is
group-theoretical if, and only if, its Drinfeld center $Z(\CC)$ is group-theoretical.
To see this, recall that the class of group-theoretical categories is closed
under tensor product, taking the opposite category, and taking duals \cite{ENO}.
Also recall that a full fusion subcategory of a group-theoretical category
is group-theoretical \cite[Proposition 8.44 (i)]{ENO}. The assertion in the
second sentence above now follows from the fact that $Z(\CC)$ is dual to
$\CC \boxtimes \CC^{\text{op}}$ \cite[Proposition 2.2]{O2}.
\end{remark}

Let $\chi$ be a nondegenerate symmetric bilinear form on an abelian
group $A$. A subgroup $L\subset A$ is {\em Lagrangian} if $L=L^\perp$
with respect to the inner product on $A$ given by $\chi$.
It was shown in \cite{GNN} that the category $\TY(A,\chi,\tau)$ is
group-theoretical if, and only if, $A$ contains a Lagrangian subgroup.
This (together with Remark \ref{C gt iff D(C) gt}) establishes the following proposition.

\begin{prop}
\label{DTY g-t}
If $A$ contains a Lagrangian subgroup, then $D\TY(A,\chi,\tau)$
is group-theoretical, and hence has property $\mathbf{F}$.
\end{prop}

\begin{example}
(i) Let $n$ be any positive integer and let $\xi \in \mathbb{C}$ be a primitive $n$-th
root of unity. Define a nondegenerate symmetric bilinear form $\chi$
on $\Z_n \times \Z_n$:
$$
\chi : (\Z_n \times \Z_n)
\times (\Z_n \times \Z_n)
\to \mathbb{C}^\times: ((x_1, x_2), (y_1, y_2)) \mapsto \xi^{x_1y_2+y_1x_2}.
$$
Then $\Z_n \times \Z_n$ contains
a Lagrangian subgroup (for example, $\Z_n \times \{0\}$).
Therefore, $D\TY(\Z_n \times \Z_n,\chi,\tau)$
has property $\mathbf{F}$ by Proposition \ref{DTY g-t}.\\

(ii) Let $A$ be an abelian group of order $2^{2t}$ and let $\chi$ be any
nondegenerate symmetric bilinear form on $A$. Then $A$ contains a Lagrangian
subgroup. Therefore, $D\TY(A,\chi,\tau)$
has property $\mathbf{F}$ by Proposition \ref{DTY g-t}.\\

(iii) Let $n$ be any positive integer. Let $\chi$ be {\em any} nondegenerate
symmetric bilinear form on $\Z_{n^2}$. Then $\Z_{n^2}$ contains a Lagrangian
subgroup: let $x$ be a generator of $\Z_{n^2}$, then the subgroup
$\langle x^n \rangle \leq \Z_{n^2}$ is Lagrangian. Therefore, $D\TY(\Z_{n^2},\chi,\tau)$
has property $\mathbf{F}$ by Proposition \ref{DTY g-t}.
\end{example}

\begin{remark}
The weakly integral categories $\CC(B_r)$ and $\CC(D_r)$ seem to be related to the weakly integral categories $D\TY(A,\chi,\tau)$.  One can show that $D\TY(A,\chi,\tau)$ for $|A|$ odd decomposes as a tensor product of a pointed modular category of rank $|A|$ and a modular category having the same fusion rules as $\CC(B_r)$ with $2r+1=|A|$ (note that $D\TY(A,\chi,\tau)$ has rank $\frac{|A|(|A|+7)}{2}$ so that $\frac{|A|+7}{2}=r+4$ which is the rank of $\CC(B_r)$).  It seems likely that $\CC(B_r)$ is equivalent to a subcategory of $D\TY(A,\chi,\tau)$ for some choice of $\chi$ and $\tau$.  The relationship with $\CC(D_r)$ is less clear, but it would be interesting to determine some precise equivalences.
\end{remark}

\end{document}